\documentclass[oneside,final,11pt]{amsart}

\usepackage{dsfont}

\newcommand{\R}{\mathds R}
\newcommand{\dd}{\mathrm d}
\newcommand{\EP}{$c_0$EP}
\newcommand{\DA}{\mathrm{DA}}
\newcommand{\anul}{\mathrm o}
\newcommand{\upstar}[1]{{#1}^{\raise1pt\hbox{$\scriptscriptstyle*$}}}
\newcommand{\chilow}[1]{\chi_{\lower2pt\hbox{$\scriptstyle#1$}}}

\DeclareMathOperator{\Ker}{Ker}
\DeclareMathOperator{\der}{der}
\DeclareMathOperator{\io}{io}

\title[On extensions of $c_0$-valued operators]{On extensions of $\mathbf{c_0}$-valued operators}

\author{Claudia Correa}
\thanks{The first author is sponsored by FAPESP (Processo n.\ 2010/17853-8).}
\address{Departamento de Matem\'atica,\hfill\break\indent Universidade de S\~ao Paulo, Brazil}
\email{claudiac.mat@gmail.com}
\author{Daniel V. Tausk}
\address{Departamento de Matem\'atica,\hfill\break\indent Universidade de S\~ao Paulo, Brazil}
\email{tausk@ime.usp.br} \urladdr{http://www.ime.usp.br/\~{}tausk}

\date{November 26th, 2012}

\begin{document}

\theoremstyle{plain}\newtheorem{teo}{Theorem}[section]
\theoremstyle{plain}\newtheorem{prop}[teo]{Proposition}
\theoremstyle{plain}\newtheorem{lem}[teo]{Lemma}
\theoremstyle{plain}\newtheorem{cor}[teo]{Corollary}
\theoremstyle{definition}\newtheorem{defin}[teo]{Definition}
\theoremstyle{remark}\newtheorem{rem}[teo]{Remark}
\theoremstyle{plain} \newtheorem{assum}[teo]{Assumption}
\theoremstyle{definition}\newtheorem{example}[teo]{Example}

\begin{abstract}
We study pairs of Banach spaces $(X,Y)$, with $Y\subset X$, for which the thesis of Sobczyk's theorem holds,
namely, such that every bounded $c_0$-valued operator defined in $Y$ extends to $X$. We are mainly concerned with
the case when $X$ is a $C(K)$ space and $Y\equiv C(L)$ is a Banach subalgebra of $C(K)$. The main
result of the article states that, if $K$ is a compact line and $L$ is countable, then every bounded $c_0$-valued operator defined in
$C(L)$ extends to $C(K)$.
\end{abstract}

\maketitle

\begin{section}{Introduction}

In this article we introduce a definition that is relevant for the investigation of two classical problems in the
theory of Banach spaces: the problem of extension of bounded operators and the problem of complementation of copies
of $c_0$ in Banach spaces. Given a Banach space $X$, we will say that a closed subspace $Y$
of $X$ has the {\em $c_0$-extension property\/} (briefly: \EP) in $X$ if every bounded operator $T:Y\to c_0$ admits
a bounded ($c_0$-valued) extension to $X$. Of course, $Y$ has the \EP\ in $X$ if, for example, $Y$ is complemented in $X$. Our main result (Theorem~\ref{thm:main}) concerns
the \EP\ for certain Banach subalgebras of a space $C(K)$. More precisely, it states that
a Banach subalgebra $C(L)$ of $C(K)$ always has the \EP\ in $C(K)$ when $L$ is countable and $K$ is a compact line. By a {\em compact line\/} we mean a linearly ordered set which is compact in the order topology. The problem
of complementation of Banach subalgebras of $C(K)$ spaces, with $K$ a compact line, was studied in a recent
article \cite{KK}.

Recall that $\ell_\infty$ is an {\em injective\/} Banach space,
i.e., for any Banach space $X$ and any closed subspace $Y$ of $X$, every bounded operator
$T:Y\to\ell_\infty$ admits a bounded extension to $X$. In particular, every isomorphic copy of $\ell_\infty$
in a Banach space is complemented. It is well-known \cite{Phillips} that the space $c_0$ is not injective:
namely, $c_0$ is not complemented in $\ell_\infty$. However, the celebrated theorem of Sobczyk
\cite{Sobczyk}, says that
$c_0$ is {\em separably injective}, i.e., if $X$ is a separable Banach space
and $Y$ is a closed subspace of $X$ then every bounded operator $T:Y\to c_0$ admits a bounded extension
to $X$. Using our terminology, this means that every closed subspace of a separable Banach space $X$ has
the \EP\ in $X$. In particular, every isomorphic copy of $c_0$ in a separable Banach space
is complemented. A Banach space in which every isomorphic copy of $c_0$ is complemented is often called a {\em Sobczyk's space}.

There are many possible directions in which one can search for generalizations
of Sobczyk's theorem: for instance, \cite{Molto, Patterson} one can look for classes of nonseparable Banach spaces in which
every isomorphic (or every isometric) copy of $c_0$ is complemented. One can also \cite{ArgyrosLondon}
replace the space $c_0$ with $c_0(\Gamma)$, where $\Gamma$ is an uncountable set, or \cite{c0sum} replace
$c_0$ with its vector-valued versions (i.e., $c_0$-type sums of Banach spaces).

Our approach is to investigate pairs of Banach spaces $(X,Y)$, with $Y\subset X$, for which the thesis of Sobczyk's theorem
holds. For instance, a simple adaptation (Proposition~\ref{thm:propVeech}) of Veech's proof of Sobczyk's
theorem \cite{Veech} shows that $Y$ has the \EP\ in $X$ if $X/Y$ is separable or if $X$ is weakly
compactly generated (WCG). We will say that $X$ has the (resp., {\em separable}) {\em $c_0$-extension property\/} if every (resp.,
separable) closed subspace of $X$ has the \EP\ in $X$. Thus, every WCG space has the \EP.
Note that if a separable Banach subspace $Y$ of a Banach space $X$ has the \EP\ in $X$ then every isomorphic copy
of $c_0$ in $Y$ is complemented in $X$. In particular, if $X$ has the separable \EP\ then $X$ is Sobczyk.

Another property of Banach spaces that is closely related to the $c_0$-extension property is the so called
separable complementation property (SCP). Recall that a Banach space $X$ is said to have the {\em separable complementation property\/} if every separable
Banach subspace of $X$ is contained in a complemented separable subspace of $X$. Clearly, if
$X$ has the SCP then $X$ has the separable \EP.

Of course, for any fixed Banach space $W$, one could define a ``$W$-extension property'', by replacing
$c_0$ with $W$ in the definition of $c_0$-extension property. For instance, in \cite{Zippin}, a Banach subspace $Y$ of $X$ is said
to be {\em almost complemented\/} in $X$ if, for every compact Hausdorff space $K$, every bounded
operator $T:Y\to C(K)$ admits a bounded extension to $X$. Since $c_0\cong C[0,\omega]$,
it is clear that if $Y$ is almost complemented in $X$ then $Y$ has the \EP\ in $X$.

The case $W=c_0$ has a nice feature: since $c_0$-valued bounded operators can be identified with weak*-null sequences of linear
functionals, we have that a Banach subspace $Y$ of $X$ has the \EP\ in $X$ if and only if every
weak*-null sequence $(\alpha_n)_{n\ge1}$ in $Y^*$ extends to a weak*-null sequence $(\tilde\alpha_n)_{n\ge1}$
in $X^*$ (i.e., $\tilde\alpha_n$ extends $\alpha_n$, for all $n\ge1$). With such formulation in terms
of linear functionals, the $c_0$-extension property for certain given Banach spaces
can be investigated using concrete representations
of their dual spaces and of weak*-convergence in those dual spaces
(see proof of Proposition~\ref{thm:propnotc0EP}). We note also
that a related (weaker) property has been investigated in \cite{JesusWeak, WZQ}: a Banach subspace
$Y$ of $X$ is said to be {\em weak*-extensible\/} in $X$ if every weak*-null sequence in $Y^*$ admits
a {\em subsequence\/} having a weak*-null extension to a sequence in $X^*$.

As mentioned earlier, we are particularly interested in studying the \EP\ for spaces of the form $C(K)$. (As
usual, $C(K)$ denotes the Banach space of real-valued continuous functions on the compact Hausdorff space $K$, endowed with the
supremum norm.) In order to study the question of whether a separable Banach subspace $X$ of a $C(K)$ space has the \EP\ in $C(K)$, one can look at the (separable) Banach subalgebra spanned by $X$: namely, if such Banach subalgebra
has the \EP\ in $C(K)$ then so does $X$. In particular, $C(K)$ has the separable \EP\ if and only if every separable Banach
subalgebra of $C(K)$ has the \EP\ in $C(K)$.
Thus, we are mainly interested in understanding the \EP\ for separable Banach subalgebras
in $C(K)$ spaces.

Given compact Hausdorff spaces $K$, $L$ and a continuous map $\phi:K\to L$, we denote by $\phi^*:C(L)\to C(K)$
the composition map $f\mapsto f\circ\phi$. The map $\phi^*$ is a Banach algebra homomorphism; if $\phi$ is onto then
$\phi^*$ is also an isometric embedding and therefore the range $\phi^*C(L)$ of $\phi^*$ can be identified with $C(L)$.
Recall that all Banach subalgebras (with unity) of $C(K)$ are of the form $\phi^*C(L)$, for some continuous surjective
map $\phi:K\to L$.

In \cite[Lemma~2.7]{KK} it is given a necessary and sufficient condition
for the complementation of $\phi^*C(L)$ in $C(K)$, if both $K$ and $L$ are compact lines and $\phi:K\to L$
is a continuous {\em increasing\/} surjection. It would be interesting to have a characterization
of Banach subalgebras $\phi^*C(L)$ of $C(K)$, with $K$ a compact line, having the \EP\ in $C(K)$.
Our main result Theorem~\ref{thm:main} is a step in this direction: we establish that $\phi^*C(L)$ always has the
\EP\ in $C(K)$ if $L$ is countable. We note that, under the same assumptions, it is not always true that $\phi^*C(L)$ is complemented in $C(K)$,
even when $\phi$ is increasing (Example~\ref{exa:bigode}).

In Theorem~\ref{thm:main}, we do not have to assume that the map $\phi$ be increasing.
In fact, under the assumption that $\phi$ be increasing, the proof that $\phi^*C(L)$ has the \EP\ in $C(K)$
is much simpler and we state this weaker result as Corollary~\ref{thm:casofacil}.
We observe that the assumption that $L$ be countable cannot be replaced with the weaker
assumption that $L$ be metrizable: namely, in Proposition~\ref{thm:propnotc0EP}, we give an example of a continuous
increasing surjection $\phi:K\to L$, with $L$ metrizable, such that $\phi^*C(L)$
does not have the \EP\ in $C(K)$.

\vspace{5pt}

The key result used in the proof of Theorem~\ref{thm:main} is Theorem~\ref{thm:juntaLi} which states, roughly,
that the class of spaces $L$ for which the thesis of Theorem~\ref{thm:main} holds is closed under the operation
of taking the Alexandroff compactification of topological sums.
It turns out that, using Theorem~\ref{thm:juntaLi}, one can obtain a stronger version of Theorem~\ref{thm:main}: namely, the assumption that $L$ be countable can be replaced with the (weaker) assumption that $L$ be scattered and hereditarily paracompact (Remark~\ref{thm:herparacomp}).

\vspace{5pt}

The article is organized as follows. In Section~\ref{sec:prelim} we prove a slight generalization of Sobczyk's theorem (Proposition~\ref{thm:propVeech})
and a technical lemma (Lemma~\ref{thm:technical}) that will be used in Section~\ref{sec:main}. In Subsection~\ref{sub:CK},
we focus on Banach spaces of the form $C(K)$, giving a sufficient condition for a subspace to have the
\EP\ in $C(K)$ (Lemma~\ref{thm:resultadodoF}). Finally, Section~\ref{sec:main} is devoted to the proof of the main result of this article
(Theorem~\ref{thm:main}). The proof of Theorem~\ref{thm:juntaLi}, which is the hard work, is presented in Subsection~\ref{sub:hell}.

\end{section}

\begin{section}{The $c_0$-extension property}
\label{sec:prelim}

In what follows, $X$ and $Y$ always denote (real) Banach spaces. We denote by $\mathcal B(X,Y)$ the space
of bounded linear operators $T:X\to Y$.

\begin{defin}\label{thm:defc0EP}
Given $\lambda\ge1$, we say that a closed subspace $Y$ of $X$ has the {\em $c_0$-extension property with constant $\lambda$\/} (briefly: $\lambda$-\EP)
in $X$ if every bounded operator $T\in\mathcal B(Y,c_0)$ admits an extension $T'\in\mathcal B(X,c_0)$ with $\Vert T'\Vert\le\lambda\Vert T\Vert$.
We say that $X$ has the (resp., {\em separable}) {\em $c_0$-extension property with constant $\lambda$\/}
if every (resp., separable) closed subspace of $X$ has the $\lambda$-\EP\ in $X$.
\end{defin}

We note that if $Y$ has the \EP\ in $X$ then $Y$ has the $\lambda$-\EP\ in $X$ for some $\lambda\ge1$. Namely, if $Y$ has the \EP\
in $X$ then the restriction map $r:\mathcal B(X,c_0)\to\mathcal B(Y,c_0)$ is onto and thus it induces an isomorphism
$\bar r:\mathcal B(X,c_0)/\Ker(r)\to\mathcal B(Y,c_0)$; any $\lambda>\Vert\bar r^{-1}\Vert$ does the job.

\medskip

Recall that a Banach space $X$ is {\em weakly compactly generated\/} (WCG) if it admits a weakly compact subset $K$
that is linearly dense in $X$, i.e., the span of $K$ is dense in $X$. The following generalization of Sobczyk's theorem is folkloric, but we
present the proof for the sake of completeness.
\begin{prop}\label{thm:propVeech}
Let $X$ be a Banach space. Then:
\begin{itemize}
\item[(a)] if $Y$ is a closed subspace of $X$ such that $X/Y$ is separable then $Y$ has the $2$-\EP\ in $X$;
\item[(b)] if $X$ is WCG then $X$ has the $2$-\EP.
\end{itemize}
\end{prop}
\begin{proof}
%
In order to prove (a), let $(\alpha_n)_{n\ge1}$ be a weak*-null sequence in $Y^*$. Without loss of generality,
we assume $\sup_{n\ge1}\Vert\alpha_n\Vert=1$. For each $n\ge1$, let $\tilde\alpha_n\in X^*$ be a Hahn--Banach
extension of $\alpha_n$. It is sufficient to find a sequence $(\beta_n)_{n\ge1}$ in $F=B_{\upstar X}\cap Y^\anul$
such that $(\tilde\alpha_n-\beta_n)_{n\ge1}$ is weak*-null. Note that every weak*-cluster point of
$(\tilde\alpha_n)_{n\ge1}$ is in $F$. Let $E$ be a countable subset of $X$ that is mapped
by the quotient map onto a dense subset of $X/Y$. Since $E$ is countable, there exists a
translation-invariant pseudometric $d$ on $X^*$ that induces the topology of convergence at the points of $E$.
The distance function $d(\cdot,F)$ is
$d$-continuous and therefore weak*-continuous. Since $B_{\upstar X}$ is weak*-compact and every weak*-cluster
point of $(\tilde\alpha_n)_{n\ge1}$ is in the zero-set of the map $d(\cdot,F)$, it follows
that $d(\tilde\alpha_n,F)\to0$. Pick $\beta_n\in F$ with $d(\tilde\alpha_n,\beta_n)<d(\tilde\alpha_n,F)+\frac1n$.
Then $d(\tilde\alpha_n-\beta_n,0)\to0$, i.e., $(\tilde\alpha_n-\beta_n)_{n\ge1}$ converges to zero at the points
of $E$. Since it also converges to zero at the points of $Y$ and since $E\cup Y$ is linearly dense in $X$,
it follows that the bounded sequence $(\tilde\alpha_n-\beta_n)_{n\ge1}$ is weak*-null.

To prove (b), let $Y$ be a closed subspace of $X$, $T:Y\to c_0$ be a bounded operator, and
$S:X\to\ell_\infty$ be an $\ell_\infty$-valued extension of $T$ with $\Vert S\Vert=\Vert T\Vert$.
Since $X/\Ker(S)$ admits a bounded linear injection to $\ell_\infty$, it has a weak*-separable dual; then $X/\Ker(S)$
is separable, because it is a WCG space with a weak*-separable dual. In particular, $X/S^{-1}[c_0]$ is
separable and, by (a), the map $S\vert_{S^{-1}[c_0]}$ has a $c_0$-valued extension $T':X\to c_0$
with $\Vert T'\Vert\le2\Vert S\Vert$. Clearly, $T'$ extends $T$ and $\Vert T'\Vert\le2\Vert T\Vert$.
\end{proof}

Note that having the $\lambda$-\EP\ is a hereditary property (yet WCG is not \cite{WCGnother}) and thus
Proposition~\ref{thm:propVeech} implies that every Banach subspace of a WCG Banach space has the $2$-\EP.
Also, in view of Proposition~\ref{thm:propVeech}, it is natural to ask whether $Y$ has the \EP\ in $X$ if $X/Y$ is
WCG. We will see in Remark~\ref{thm:remquocWCG} below that the answer is negative.

\vspace{5pt}

Clearly, a closed subspace $Y$ of $X$ has the $\lambda$-\EP\ in $X$ if and only if every weak*-null sequence
$(\alpha_n)_{n\ge1}$ in $Y^*$ extends to a weak*-null sequence $(\tilde\alpha_n)_{n\ge1}$ in $X^*$ with
$\sup_{n\ge1}\Vert\tilde\alpha_n\Vert\le\lambda\,\sup_{n\ge1}\Vert\alpha_n\Vert$. In the next lemma,
we show that, if $Y$ has the $\lambda$-\EP\ in $X$, then the weak*-null extension $(\tilde\alpha_n)_{n\ge1}$
can in fact be chosen with $\Vert\tilde\alpha_n\Vert\le\lambda\Vert\alpha_n\Vert$, for each $n$.
Moreover, in order to establish that $Y$ has the $\lambda$-\EP\ in $X$, it is sufficient to consider weak*-null sequences $(\alpha_n)_{n\ge1}$ in $Y^*$ that are normalized.
We say that a family in a Banach space is {\em normalized\/} if all of its members have norm $1$.
\begin{lem}\label{thm:technical}
Let $X$ be a Banach space, $Y$ be a closed subspace of $X$ and $\lambda\ge1$. The following statements are equivalent:
\begin{itemize}
\item[(a)] $Y$ has the $\lambda$-\EP\ in $X$;
\item[(b)] every normalized weak*-null sequence $(\alpha_n)_{n\ge1}$ in $Y^*$ extends to a weak*-null
sequence $(\tilde\alpha_n)_{n\ge1}$ in $X^*$ with $\Vert\tilde\alpha_n\Vert\le\lambda$, for all $n\ge1$;
\item[(c)] every weak*-null sequence $(\alpha_n)_{n\ge1}$ in $Y^*$ extends to a weak*-null sequence
$(\tilde\alpha_n)_{n\ge1}$ in $X^*$ with $\Vert\tilde\alpha_n\Vert\le\lambda\Vert\alpha_n\Vert$, for all $n\ge1$.
\end{itemize}
\end{lem}
\begin{proof}
The implications (a)$\Rightarrow$(b) and (c)$\Rightarrow$(a) are trivial, so it remains to prove
(b)$\Rightarrow$(c). Let $(\alpha_n)_{n\ge1}$ be a weak*-null sequence in $Y^*$. Without loss of generality,
assume $\sup_{n\ge1}\Vert\alpha_n\Vert=1$. Set $\beta_n=\frac{\alpha_n}{\Vert\alpha_n\Vert}$ when $\alpha_n\ne0$.
For each integer $k\ge1$, let $N_k$ denote the set of all integers
$n\ge1$ such that $\frac1{k+1}<\Vert\alpha_n\Vert\le\frac1k$. If, for a given $k$, the set $N_k$ is infinite,
then $(\beta_n)_{n\in N_k}$ is a normalized weak*-null sequence in $Y^*$; by (b), it
extends to a weak*-null sequence $(\tilde\beta_n)_{n\in N_k}$ in $X^*$ with $\Vert\tilde\beta_n\Vert\le\lambda$,
for all $n\in N_k$.

We now define the sequence $(\tilde\alpha_n)_{n\ge1}$
in $X^*$ as follows. If $\alpha_n=0$, set $\tilde\alpha_n=0$. If $n\in N_k$ and $N_k$ is finite, let $\tilde\alpha_n\in X^*$
be a Hahn--Banach extension of $\alpha_n$. If $n\in N_k$ and $N_k$ is infinite, set
$\tilde\alpha_n=\Vert\alpha_n\Vert\tilde\beta_n$. Clearly, $\tilde\alpha_n$ extends $\alpha_n$
and $\Vert\tilde\alpha_n\Vert\le\lambda\Vert\alpha_n\Vert$, for all $n\ge1$. Moreover, if $N_k$ is infinite,
$(\tilde\alpha_n)_{n\in N_k}$ is weak*-null. Let us show that $(\tilde\alpha_n)_{n\ge1}$
is weak*-null. Fix $x\in X$ with $\Vert x\Vert=1$ and let $\varepsilon>0$ be given. Choose $k_0\ge1$ with
$\frac\lambda{k_0}<\varepsilon$. For $k\ge k_0$ and $n\in N_k$, we have $\vert\tilde\alpha_n(x)\vert
\le\lambda\Vert\alpha_n\Vert<\varepsilon$. Clearly, for $n\in\bigcup_{k<k_0}N_k$ sufficiently large,
we have $\vert\tilde\alpha_n(x)\vert<\varepsilon$. This concludes the proof.
\end{proof}

\begin{cor}\label{thm:eraremark}
Let $X$ be a Banach space, $Y$ be a closed subspace of $X$ and $\lambda\ge1$.
Assume that every normalized weak*-null sequence $(\alpha_n)_{n\ge1}$ in $Y^*$ extends to a weak*-null sequence
$(\tilde\alpha_n)_{n\ge1}$ in $X^*$ such that:
\begin{equation}\label{eq:limsup}
\limsup_{n\to+\infty}\Vert\tilde\alpha_n\Vert\le\lambda.
\end{equation}
Then, for any $\varepsilon>0$, $Y$ has the $(\lambda+\varepsilon)$-\EP\ in $X$.
\end{cor}
\begin{proof}
Inequality \eqref{eq:limsup} implies that $\Vert\tilde\alpha_n\Vert\le\lambda+\varepsilon$, for all $n$ greater
than some $n_0$. Now note that we can replace $\tilde\alpha_n$ with a Hahn--Banach extension of $\alpha_n$,
for $n\le n_0$.
\end{proof}

\subsection{The $\mathbf{c_0}$-extension property for $\mathbf{C(K)}$ spaces}\label{sub:CK}
In what follows, $K$ and $L$ denote compact Hausdorff spaces. We always identify the dual space
of $C(K)$ with the space $M(K)$ of finite countably-additive signed regular Borel measures on $K$,
endowed with the total variation norm $\Vert\mu\Vert=\vert\mu\vert(K)$. Given a point $p\in K$,
we denote by $\delta_p\in M(K)$ the probability measure with support $\{p\}$. Note that if $\phi:K\to L$
is a continuous map then the adjoint of the composition map $\phi^*:C(L)\to C(K)$ is the push-forward
map $\phi_*:M(K)\to M(L)$ defined by $\phi_*(\mu)(B)=\mu\big(\phi^{-1}[B]\big)$, for any $\mu\in M(K)$
and any Borel subset $B$ of $L$. When $\phi$ is onto, we can identify $C(L)$ with the subspace $\phi^*C(L)$ of $C(K)$
and then an extension of $\mu\in M(L)\equiv C(L)^*$ to $C(K)$ is identified with a measure $\tilde\mu\in M(K)$
such that $\phi_*\tilde\mu=\mu$.

\vspace{5pt}

The following lemma gives a simple positive criterion for a Banach subspace of $C(K)$ to have
\EP\ in $C(K)$. Recall that $K$ is said to be an {\em Eberlein compact\/} if $K$ is homeomorphic to a weakly
compact subset of some Banach space. Moreover, $K$ is Eberlein if and only if $C(K)$ is WCG.
\begin{lem}\label{thm:resultadodoF}
Let $K$, $L$ be compact Hausdorff spaces, $\phi:K\to L$ be a continuous map and assume that there exists
an Eberlein compact subset $F$ (for instance, a metrizable closed subset) of $K$ such that $\phi\vert_F$ is onto.
Then every closed subspace of $\phi^*C(L)$ has the $2$-\EP\ in $C(K)$.
\end{lem}
\begin{proof}
Consider a closed subspace $\phi^*[X]$ of $\phi^*C(L)$, where $X$ is a closed subspace of $C(L)$.
We have to show that for any $T\in\mathcal B(X,c_0)$, there exists $T'\in\mathcal B\big(C(K),c_0\big)$ with
$T'\circ\phi^*\vert_X=T$ and $\Vert T'\Vert\le2\Vert T\Vert$. Since $\phi\vert_F$ is onto, $(\phi\vert_F)^*$ embeds
$X$ isometrically into $C(F)$; by Proposition~\ref{thm:propVeech}, since $C(F)$ is WCG, for any $T\in\mathcal B(X,c_0)$
there exists $T_1\in\mathcal B\big(C(F),c_0\big)$ such that $T_1\circ(\phi\vert_F)^*\vert_X=T$ and
$\Vert T_1\Vert\le2\Vert T\Vert$. To conclude the proof, set $T'=T_1\circ\rho_F$, where $\rho_F:C(K)\to C(F)$
denotes the restriction map.
\end{proof}

The following corollary is weaker than our main result Theorem~\ref{thm:main} (except for the fact that we get $2$-\EP\
in the thesis of the corollary and $(2+\varepsilon)$-\EP\ in the thesis of Theorem~\ref{thm:main}).
However, since its proof is much simpler, we believe it's interesting to state it here.
\begin{cor}\label{thm:casofacil}
Let $K$ be a compact line, $L$ be a countable compact Hausdorff space and $\phi:K\to L$ be a continuous surjection.
Assume that, for each $p\in L$, $\phi^{-1}(p)$ is a countable union of intervals of $K$ (this is the case,
for instance, if $L$ is a compact line and $\phi$ is increasing). Then $\phi^*C(L)$ has the $2$-\EP\ in $C(K)$.
\end{cor}
\begin{proof}
For each $p\in L$, $\phi^{-1}(p)$ can be written as a countable disjoint union of closed intervals of $K$
(namely, the convex components of $\phi^{-1}(p)$); denote by $F_p$ the set of all endpoints of such closed
intervals and set $F=\bigcup_{p\in L}F_p$. Obviously, $F$ is countable and $\phi\vert_F$ is onto. Moreover,
$F$ is closed, since $K\setminus F$ is a union of open intervals of $K$.
\end{proof}

Using Lemma~\ref{thm:resultadodoF}, we easily obtain a class of compact spaces $K$ such that the space $C(K)$
has the separable \EP\ (and is, in particular, Sobczyk).
Recall that $K$ is {\em $\aleph_0$-monolithic\/} if every separable subspace of $K$ is second countable.
\begin{cor}
If $K$ is $\aleph_0$-monolithic then $C(K)$ has the separable $2$-\EP.
\end{cor}
\begin{proof}
A closed separable subspace of $C(K)$ spans a separable Banach subalgebra $\phi^*C(L)$ of $C(K)$,
where $\phi:K\to L$ is continuous onto and $L$ is metrizable. Let $F$ be the closure of a countable subset of $K$
that is mapped by $\phi$ onto a dense subset of $L$. Then $F$ is metrizable and we are done.
\end{proof}

\begin{rem}
We observe that, for example, every scattered compact line is $\aleph_0$-monolithic. Namely, if $K$
is a scattered compact line and $F$ were a separable nonmetrizable closed subset of $K$, then
$F$ would be a separable nonmetrizable compact line and therefore there would exist a continuous surjection
$F\to[0,1]$ (\cite[Proof of Lemma~2.5]{KK}). This contradicts the fact that $F$ is scattered.
\end{rem}

Recall that the {\em double arrow space\/} is the set $\DA=[0,1]\times\{0,1\}$ endowed with the lexicographic order
and the order topology. The double arrow space is a compact line and the first projection $\pi_1:\DA\to[0,1]$ is a
continuous increasing surjection. It has long been known (\cite[Example~2]{Corson}, \cite[pg.\ 6]{Kubis}) that
$\pi_1^*C[0,1]$ is not complemented in $C(\DA)$. In Proposition~\ref{thm:propnotc0EP} below
we obtain a new proof of this fact, illustrating that the notion of $c_0$-extension property provides a technique for proving
that a subspace of a Banach space is not complemented. Proposition~\ref{thm:propnotc0EP} also shows that the thesis of
Corollary~\ref{thm:casofacil} does not hold if one only assumes that $L$ be metrizable (even in the case when $L$ is a compact line
and $\phi$ is increasing).

\begin{prop}\label{thm:propnotc0EP}
The subalgebra $\pi_1^*C[0,1]$ does not have the \EP\ in the Banach space $C(\DA)$.
\end{prop}
\begin{proof}
Given $\mu\in M(\DA)$, $\nu\in M[0,1]$, we define maps $F_\mu:[0,1]\to\R$ and $G_\nu:[0,1]\to\R$,
by setting:
\[F_\mu(t)=\mu[(0,0),(t,0)],\quad G_\nu(t)=\nu[0,t],\qquad t\in[0,1].\]
The maps $F_\mu$ and $G_\nu$ are differences of increasing functions and therefore have bounded variation; in particular,
they have at most a countable number of discontinuity points. If $\nu=(\pi_1)_*\mu$,
then $G_\nu(t)=F_\mu(t^+)$, for all $t\in\left[0,1\right[$. In particular, $F_\mu$ and $G_\nu$ differ at most
at a countable subset of $[0,1]$. We claim that if a weak*-null sequence $(\nu_n)_{n\ge1}$ in $M[0,1]$ admits
a weak*-null extension $(\mu_n)_{n\ge1}$ in $M(\DA)$ then $(G_{\nu_n})_{n\ge1}$ converges to zero pointwise
outside a countable subset of $[0,1]$. Namely, observe that $[(0,0),(t,0)]$ is clopen and therefore $(F_{\mu_n})$ converges
to zero pointwise. Moreover, since $(\pi_1)_*\mu_n=\nu_n$, the maps $F_{\mu_n}$ and $G_{\nu_n}$ differ at most
at a countable subset of $[0,1]$. This proves the claim.

In order to conclude the proof of the proposition, we simply exhibit a weak*-null sequence $(\nu_n)_{n\ge1}$
in $M[0,1]$ such that $(G_{\nu_n})_{n\ge1}$ does not converge to zero pointwise outside a countable subset of $[0,1]$.
For instance, pick a sequence of intervals $[a_n,b_n]\subset[0,1]$ such that
$\lim_{n\to+\infty}(b_n-a_n)=0$ and set $\nu_n=\delta_{a_n}-\delta_{b_n}$; the uniform continuity of each $f\in C[0,1]$ implies that
the sequence $(\nu_n)_{n\ge1}$ is weak*-null. Note that $G_{\nu_n}$ equals the characteristic function
$\chilow{\left[a_n,b_n\right[}$ and that the intervals can be chosen in a way that $(G_{\nu_n})_{n\ge1}$ does
not converge to zero at any point of $\left[0,1\right[$.
\end{proof}

\begin{rem}\label{thm:remquocWCG}
It is well-known that the map
\[C(\DA)\ni f\longmapsto\big(f(t,1)-f(t,0)\big)_{t\in[0,1]}\in c_0[0,1]\]
induces an isomorphism between $C(\DA)/\pi_1^*C[0,1]$ and $c_0[0,1]$. Thus, the quotient $C(\DA)/\pi_1^*C[0,1]$ is WCG.
In view of Proposition~\ref{thm:propnotc0EP}, we thus have an example of a $C(K)$ space and a separable Banach
subalgebra $\phi^*C(L)$ such that $C(K)/\phi^*C(L)$ is WCG, but $\phi^*C(L)$ does not have the \EP\ in $C(K)$.
\end{rem}

We finish the section by showing (Example~\ref{exa:bigode}) that there exists a continuous increasing surjection $\phi:K\to L$ such that $K$ and $L$ are
countable compact lines and $\phi^*C(L)$ is not complemented in $C(K)$.
To prove that $\phi^*C(L)$ is not complemented in $C(K)$ we will use \cite[Lemma~2.7]{KK}.

We recall some definitions. Let $X$ be a linearly
ordered set and $A$ be a subset of $X$. A point $x\in X$ is a {\em right limit point\/} of $A$ (relatively to $X$) if $x$ is not the maximum
of $X$ and for every $y\in X$ with $y>x$ we have $\left]x,y\right[\cap A\ne\emptyset$. Similarly, one defines left limit points. We denote
by $\der(A,X)$ the set of points $a\in A$ that are both right and left limits of $A$ relatively to $X$. The points of $\der(X,X)$ are called
{\em internal points\/} of $X$ (in accordance with the terminology of \cite{KK}). Recursively, we define $\der_n(A,X)$ by setting
$\der_0(A,X)=A$ and $\der_{n+1}(A,X)=\der\!\big(\!\der_n(A,X),X\big)$. The {\em internal order\/} of $A$ (relatively to $X$)
is defined by $\io(A,X)=n-1$, where $n$ is the least natural number with $\der_n(A,X)=\emptyset$; we set $\io(A,X)=+\infty$ if $\der_n(A,X)\ne\emptyset$,
for all $n$.

In \cite[Lemma~2.7,~(2)]{KK} the authors prove that, given compact lines $K$, $L$ and a continuous increasing surjection $\phi:K\to L$,
the Banach subalgebra $\phi^*C(L)$ of $C(K)$ is complemented if and only if $\io(Q,L)<+\infty$, where $Q$ is the set of points $p\in\der(L,L)$ such that $\phi^{-1}(p)$
has more than one point.

\begin{rem}\label{thm:remKL}
The following observation is useful for constructing compact lines. Given a compact line $K$ and a family of nonempty compact lines $(L_x)_{x\in K}$,
the set $\bigcup_{x\in K}\big(\{x\}\times L_x\big)$ endowed with the lexicographic order is also a compact line.
\end{rem}

\begin{example}\label{exa:bigode}
Let $B=\big\{{\pm\frac1k}:k=1,2,\ldots\big\}\cup\{0\}$ be endowed with the standard order of real numbers. For each positive integer $n$,
denote by $B_n$ the set of those $(x_1,\ldots,x_n)\in B^n$ such that $x_i=0$ implies $x_{i+1}=0$, for all $i=1,\ldots,n-1$.
Endow $B_n$ with the lexicographic order. Using Remark~\ref{thm:remKL}, it follows easily by induction that $B_n$ is a compact line, for all $n$.
Moreover, for $i=1,\ldots,n-1$, we have $\der_i(B_n,B_n)=B_{n-i}\times\{0\}^i$ and thus $\io(B_n,B_n)=n$.
Denote by $L$ the disjoint union $\big(\bigcup_{n=1}^\infty B_n\big)\cup\{\infty\}$, ordered so that the elements of $B_n$ precede the elements of
$B_{n+1}$ and such that $\infty$ is the largest element of $L$. Then $L$ is a countable compact line. Since $B_n$ is a convex subset of
$L$, we have $\io(B_n,L)=\io(B_n,B_n)=n$ and therefore $\io(L,L)=+\infty$, since $\io(B_n,L)\le\io(L,L)$, for all $n$. Let $K=L\times\{0,1\}$ be
endowed with the lexicographic order and $\phi:K\to L$ denote the first projection. Then $K$ is also a countable compact line and $\phi$ is a continuous
increasing surjection. The set $Q$ is equal to $\der(L,L)$ and therefore $\io(Q,L)=+\infty$. It follows from \cite[Lemma~2.7,~(2)]{KK}
that $\phi^*C(L)$ is not complemented in $C(K)$.
\end{example}

\end{section}

\begin{section}{Main result}\label{sec:main}

In this section we prove our main result, which is the following.
\begin{teo}\label{thm:main}
Let $K$ be a compact line, $L$ be a countable compact Hausdorff space and $\phi:K\to L$ be a continuous surjection.
Then, for any $\varepsilon>0$, the subspace $\phi^*C(L)$ has the $(2+\varepsilon)$-\EP\ in $C(K)$.
\end{teo}

Recall that every nonempty countable compact Hausdorff space is homeomorphic to an ordinal segment $[0,\alpha]$,
for some $\alpha<\omega_1$. We will prove Theorem~\ref{thm:main} using induction on $\alpha$.
The hard case is when $\alpha$ is a limit ordinal. The following result takes care of that.
\begin{teo}\label{thm:juntaLi}
Let $K$ be a compact line, $L$ be a compact Hausdorff space and $\phi:K\to L$ be a continuous surjection.
Let $(L_\gamma)_{\gamma\in\Gamma}$ be a disjoint family of clopen subsets of $L$ such that
$L\setminus\big(\bigcup_{\gamma\in\Gamma}L_\gamma\big)$ has at most one point. Set $K_\gamma=\phi^{-1}[L_\gamma]$ and
$\phi_\gamma=\phi\vert_{K_\gamma}:K_\gamma\to L_\gamma$. Let $\lambda\ge2$ and assume that $(\phi_\gamma)^*C(L_\gamma)$ has the
$\lambda$-\EP\ in $C(K_\gamma)$, for all $\gamma\in\Gamma$. Then, for any $\varepsilon>0$, $\phi^*C(L)$ has
the $(\lambda+\varepsilon)$-\EP\ in $C(K)$.
\end{teo}

Let us start by proving Theorem~\ref{thm:main} using Theorem~\ref{thm:juntaLi}. The proof of Theorem~\ref{thm:juntaLi}
requires several technical lemmas and will be presented in Subsection~\ref{sub:hell} below.
\begin{proof}[Proof of Theorem~\ref{thm:main}]
For every countable ordinal $\alpha$, denote by $\mathbb P(\alpha)$ the condition which asserts that
the thesis of the theorem holds for $L=[0,\alpha]$, any compact line $K$ and any continuous surjection $\phi:K\to L$.
If $\alpha$ is finite then $\phi^*C[0,\alpha]$ has the $1$-\EP\ in $C(K)$ because a right inverse $s$ of $\phi$ is continuous
and $\phi^*\circ s^*:C(K)\to\phi^*C[0,\alpha]$ is a norm-$1$ projection. If $\alpha$ is infinite then $[0,\alpha]$
and $[0,\alpha+1]$ are homeomorphic, so $\mathbb P(\alpha)$ is equivalent to $\mathbb P(\alpha+1)$. To conclude the proof,
let $\alpha<\omega_1$ be a limit ordinal and set $L=[0,\alpha]$.
Let $(\alpha_n)_{n\ge0}$ be a strictly increasing sequence of ordinals with $\sup_{n\ge0}\alpha_n=\alpha$. Set
$L_0=[0,\alpha_0]$ and $L_n=\left]\alpha_{n-1},\alpha_n\right]$, for $n\ge1$. Each $L_n$ is homeomorphic to $[0,\beta_n]$,
for some $\beta_n<\alpha$; therefore, given $\varepsilon>0$, the subspace $(\phi_n)^*C(L_n)$ has
the $(2+\varepsilon)$-\EP\ in $C(K_n)$, where $K_n=\phi^{-1}[L_n]$ and $\phi_n=\phi\vert_{K_n}:K_n\to L_n$.
Hence, Theorem~\ref{thm:juntaLi} yields that $\phi^*C(L)$ has the $(2+2\varepsilon)$-\EP\ in $C(K)$.
\end{proof}

Theorem~\ref{thm:main} has some interesting corollaries concerning the complementation of isomorphic copies of $c_0$
in $C(K)$ spaces.
\begin{cor}
Let $K$ be a compact line and $(f_n)_{n\ge1}$ be a sequence in $C(K)$ equivalent to the canonical basis
of $c_0$. If the map:
\begin{equation}\label{eq:countablerange}
\phi:K\ni p\longmapsto\big(f_n(p)\big)_{n\ge1}\in\R^\omega
\end{equation}
has countable range
then the isomorphic copy of $c_0$ spanned by $(f_n)_{n\ge1}$ is complemented in $C(K)$.
\end{cor}
\begin{proof}
Note that the Banach subalgebra spanned by $(f_n)_{n\ge1}$ is $\phi^*C(L)$, where $L$ denotes the range of
$\phi$ (endowed with the topology induced by the product topology of $\R^\omega$).
\end{proof}

\begin{cor}
Let $K$ be a compact line and $(f_n)_{n\ge1}$ be a sequence in $C(K)$ equivalent to the canonical basis
of $c_0$. Assume that each $f_n$ has countable range and that there exists $\varepsilon>0$ such that,
for all $n\ge1$ and all $p\in K$, $\vert f_n(p)\vert\ge\varepsilon$ if $f_n(p)\ne0$. Then the isomorphic copy
of $c_0$ spanned by $(f_n)_{n\ge1}$ is complemented in $C(K)$.
\end{cor}
\begin{proof}
Let $T\in\mathcal B\big(c_0,C(K)\big)$ be the map that carries the canonical basis of $c_0$ to $(f_n)_{n\ge1}$.
Note that $\sum_{n=1}^\infty\vert f_n(p)\vert=\Vert T^*\delta_p\Vert\le\Vert T\Vert$, for all $p\in K$.
Thus, our assumptions imply that, for each $p\in K$, the sequence $\big(f_n(p)\big)_{n\ge1}$ is eventually
null; since each $f_n$ has countable range, it follows that the map \eqref{eq:countablerange} has countable range as well.
\end{proof}

\vspace{5pt}

The next proposition implies that the constant $2$ in the thesis of Theorem~\ref{thm:main} is optimal.
\begin{prop}
Let $K$ be a Boolean space having more than one limit point. Then there exists a continuous surjection
$\phi:K\to[0,\omega]$ such that $\phi^*C[0,\omega]$ does not have the $\lambda$-\EP\ in $C(K)$ for $\lambda<2$.
\end{prop}
\begin{proof}
The space $K$ can be written as a disjoint union $K=B_0\cup B_1$, with $B_0$, $B_1$ infinite clopen subsets of $K$.
For $i=0,1$, let $(B^k_i)_{k\ge0}$ be a sequence of nonempty disjoint clopen subsets of $B_i$. Define the map
$\phi$ by setting $\phi(p)=2k+i$ for $p\in B^k_i$ and $\phi(p)=\omega$ otherwise.
Consider the normalized weak*-null sequence $(v_n)_{n\ge0}$ in $\ell_1[0,\omega]\equiv C[0,\omega]^*$ defined as follows:
set $v_n(n)=\frac12$, $v_n(n+1)=-\frac12$ and $v_n(j)=0$ for $j\in[0,\omega]\setminus\{n,n+1\}$. We claim
that if $(\mu_n)_{n\ge0}$ is a weak*-null sequence in $M(K)$ such that $\phi_*\mu_n=v_n$, for all $n\ge0$, then
$\sup_{n\ge0}\Vert\mu_n\Vert\ge2$. Note that the equality $\phi_*\mu_n=v_n$ means that
$\mu_n(B^k_i)=v_n(2k+i)$ and that $\mu_n(K)=0$. Set $A_i=\bigcup_{k=0}^\infty B^k_i$,
so that $\vert\mu_n(A_i)\vert=\frac12$. Now:
\begin{align*}
\vert\mu_n\vert(B_i)&\ge\vert\mu_n(A_i)\vert+\vert\mu_n(B_i\setminus A_i)\vert=
\vert\mu_n(A_i)\vert+\vert\mu_n(B_i)-\mu_n(A_i)\vert\\
&\ge2\vert\mu_n(A_i)\vert-\vert\mu_n(B_i)\vert=1-\vert\mu_n(B_i)\vert,
\end{align*}
and:
\[\Vert\mu_n\Vert=\vert\mu_n\vert(B_0)+\vert\mu_n\vert(B_1)\ge2-\vert\mu_n(B_0)\vert-\vert\mu_n(B_1)\vert.\]
Since $B_i$ is clopen, $\lim_{n\to+\infty}\mu_n(B_i)=0$ and the claim is proved.
\end{proof}

\begin{rem}\label{thm:herparacomp}
In Theorem~\ref{thm:main}, the assumption that $L$ be countable can be weakened: namely, it is sufficient
to assume that $L$ be hereditarily paracompact and scattered. To see this, consider the class of all compact Hausdorff
spaces $L$ such that the thesis of Theorem~\ref{thm:main} holds for any compact line $K$ and any continuous surjection $\phi:K\to L$.
Theorem~\ref{thm:juntaLi} implies that such class is closed under the operation of taking the Alexandroff
compactification of topological sums. Moreover, it is known \cite[Theorem~3,~(3)]{Taras} that the smallest class of spaces
closed under such operation and containing the singleton is the class of compact Hausdorff
hereditarily paracompact scattered spaces.
\end{rem}

\subsection{Some technical lemmas and the proof of Theorem~\ref{thm:juntaLi}}\label{sub:hell}
Recall that if $F$ is a closed subset of $K$ then an {\em extension operator\/} for $F$
is a bounded operator $E_F:C(F)\to C(K)$ which is a right inverse for the restriction operator
$\rho_F:C(K)\to C(F)$. We call $E_F$ a {\em regular\/} extension operator if, in addition,
$\Vert E_F\Vert\le1$ and $E_F(\mathbf1_F)=\mathbf1_K$. If $K$ is a compact line then
every nonempty closed subset of $K$ admits a regular extension operator (\cite[Lemma~4.2]{Kubis}).
\begin{lem}\label{thm:lemaP}
Let $K$ be a compact Hausdorff space, $F$ be a closed subset of $K$ and $E_F:C(F)\to C(K)$ be
a regular extension operator. Let $(p_\theta)_{\theta\in\Theta}$ be a family of distinct points of $K\setminus F$ such
that every limit point of $\{p_\theta:\theta\in\Theta\}$ is in $F$. Then the equality:
\begin{equation}\label{eq:defP}
P(f)(\theta)=\big(f-E_F(f\vert_F)\big)(p_\theta),\quad f\in C(K),\ \theta\in\Theta,
\end{equation}
defines a bounded linear map $P:C(K)\to c_0(\Theta)$. Moreover, for every $v\in c_0(\Theta)^*\equiv\ell_1(\Theta)$,
the measure $P^*(v)\in M(K)$ satisfies:
\begin{itemize}
\item[(i)] $P^*(v)(K)=0$;
\item[(ii)] $P^*(v)$ and $\sum_{\theta\in\Theta}v(\theta)\delta_{p_\theta}$ define the same Borel measure
on $K\setminus F$;
\item[(iii)] the total variation $\vert P^*(v)\vert$ satisfies $\vert P^*(v)\vert(F)\le\Vert v\Vert$.
\end{itemize}
\end{lem}
\begin{proof}
Given $f\in C(K)$, since every limit point of $\{p_\theta:\theta\in\Theta\}$ is in the zero set
of $f-E_F(f\vert_F)$, it follows that $P(f)$ indeed belongs to $c_0(\Theta)$. From $P(\mathbf1_K)=0$
we obtain (i). It remains to prove (ii) and (iii).
Define a bounded operator
$P_1:C(K)\to\ell_\infty(\Theta)$ by setting $P_1(f)(\theta)=f(p_\theta)$.
Then $P$, regarded with counterdomain in $\ell_\infty(\Theta)$, can be written as $P=P_1-P_2$, where
$P_2=P_1\circ E_F\circ\rho_F$. Using the standard bilinear pairing between $\ell_1(\Theta)$ and $\ell_\infty(\Theta)$,
we can regard $v\in\ell_1(\Theta)$ as linear functional on $\ell_\infty(\Theta)$. We can therefore write
$P^*(v)=P_1^*(v)-P_2^*(v)$. Clearly, $P_1^*(v)=\sum_{\theta\in\Theta}v(\theta)\delta_{p_\theta}$.
To prove (ii), we now check that $P_2^*(v)$ vanishes identically on $K\setminus F$. Namely, if $i:F\to K$ denotes the inclusion
map then $\rho_F=i^*$ and thus the adjoint of $\rho_F$ is the push-forward
$i_*:M(F)\to M(K)$ that extends measures to zero. Finally, to prove (iii), note that $P_1^*(v)$ vanishes identically on $F$, and hence:
\[\vert P^*(v)\vert(F)=\vert P_2^*(v)\vert(F)\le\Vert P_2^*(v)\Vert\le\Vert v\Vert.\qedhere\]
\end{proof}

\begin{lem}\label{thm:inequalityintegral}
Let $(\mathfrak X,\mathcal A)$ be a measurable space and $\mu:\mathcal A\to\R$ be a signed measure
with $\mu(\mathfrak X)=0$. If $f:\mathfrak X\to\R$ is a bounded measurable function then:
\begin{equation}\label{eq:inequalityintegral}
\Big\vert\int_{\mathfrak X}f\,\dd\mu\Big\vert\le\frac12\,(\sup f-\inf f)\Vert\mu\Vert.
\end{equation}
\end{lem}
\begin{proof}
Since $\mu(\mathfrak X)=0$, neither side of inequality \eqref{eq:inequalityintegral} changes if we add a constant
to $f$. Thus, we can assume $\sup f=-\inf f$, in which case $\sup_{x\in\mathfrak X}\vert f(x)\vert\le\frac12(\sup f-\inf f)$.
\end{proof}

\begin{rem}
Note that if $\mu\in M(K)$ and $(U_\theta)_{\theta\in\Theta}$ is a (possibly uncountable) family of disjoint open
subsets of $K$ then:
\[\mu(B)=\sum_{\theta\in\Theta}\mu(B\cap U_\theta),\]
for any Borel subset $B$ of $U=\bigcup_{\theta\in\Theta}U_\theta$; moreover, if $f\in L^1(K,\mu)$ then:
\[\int_Uf\,\dd\mu=\sum_{\theta\in\Theta}\int_{U_\theta}f\,\dd\mu.\]
To prove these equalities, note that $\vert\mu\vert(U_\theta)=0$ for $\theta\in\Theta$ outside some countable subset
$\Theta_0$ of $\Theta$ and that, by regularity, $\mu$ vanishes identically on $\bigcup_{\theta\in\Theta\setminus\Theta_0}U_\theta$.
\end{rem}

\begin{lem}\label{thm:colamedida}
Let $K$ be a compact line and $(I_\theta)_{\theta\in\Theta}$ be a disjoint family of clopen intervals of $K$
such that $\bigcup_{\theta\in\Theta}I_\theta$ is a proper subset of $K$.
For each $\theta\in\Theta$, let $(\nu^n_\theta)_{n\ge1}$ be a weak*-null sequence in $M(I_\theta)$.
Assume that:
\begin{equation}\label{eq:bounded}
\sup_{n\ge1}\sum_{\theta\in\Theta}\Vert\nu^n_\theta\Vert<+\infty.
\end{equation}
Then, there exists a weak*-null sequence $(\nu^n)_{n\ge1}$ in $M(K)$ such that, for all $n\ge1$:
\begin{itemize}
\item[(a)] $\nu^n\vert_{I_\theta}=\nu^n_\theta$, for all $\theta\in\Theta$;
\item[(b)] $\nu^n(K)=0$;
\item[(c)] $\Vert\nu^n\Vert\le\sum_{\theta\in\Theta}\big(\Vert\nu^n_\theta\Vert+\vert\nu^n_\theta(I_\theta)\vert\big)$.
\end{itemize}
\end{lem}
\begin{proof}
Assuming without loss of generality that all $I_\theta$ are nonempty, pick $p_\theta\in I_\theta$,
for each $\theta\in\Theta$. Set $F=K\setminus\bigcup_{\theta\in\Theta}I_\theta$ and note that every limit point
of $\{p_\theta:\theta\in\Theta\}$ is in $F$. Since $F$ is a nonempty closed subset of the compact line $K$,
there exists a regular extension operator $E_F$. Define $P$ as in \eqref{eq:defP}.
Our plan is the following: we will start by defining a certain weak*-null sequence $(v^n)_{n\ge1}$ in $\ell_1(\Theta)\equiv c_0(\Theta)^*$; using $P$, this yields a weak*-null sequence of measures $P^*(v^n)\in M(K)$. Then,
since $\sum_{\theta\in\Theta}\Vert\nu^n_\theta\Vert<+\infty$, it is easily seen that there exists a unique
$\nu^n\in M(K)$ satisfying (a) and $\nu^n\vert_F=P^*(v^n)\vert_F$.

We define $v^n$ by setting $v^n(\theta)=\nu^n_\theta(I_\theta)$. By \eqref{eq:bounded}, the sequence $(v^n)_{n\ge1}$ is bounded in $\ell_1(\Theta)$ and, since $(\nu^n_\theta)_{n\ge1}$ is weak*-null, we have that $v^n(\theta)\to0$ for each $\theta$. Thus, $(v^n)_{n\ge1}$ is weak*-null. As explained above, from $v^n$ we obtain $\nu^n\in M(K)$.
Using (ii) of Lemma~\ref{thm:lemaP}, we obtain:
\[P^*(v^n)(K\setminus F)=\sum_{\theta\in\Theta}v^n(\theta)=\nu^n(K\setminus F)\]
and therefore $\nu^n(K)=P^*(v^n)(K)=0$, by (i) of Lemma~\ref{thm:lemaP}. This proves (b). To prove (c), note that:
\[\Vert\nu^n\Vert=\vert P^*(v^n)\vert(F)+\sum_{\theta\in\Theta}\Vert\nu^n_\theta\Vert,\]
and use (iii) of Lemma~\ref{thm:lemaP}.

Finally, we prove that $(\nu^n)_{n\ge1}$ is weak*-null.
Since $(\nu^n)_{n\ge1}$ is bounded and the increasing functions form a linearly dense subset of $C(K)$
(\cite[Proposition~3.2]{Kubis}), it suffices to show that $\int_Kf\,\dd\nu^n\to0$, for any continuous increasing
function $f:K\to\R$. Recalling that $\big(P^*(v^n)\big)_{n\ge1}$ is weak*-null, it suffices to show that:
\[\lim_{n\to+\infty}\int_Kf\,\dd\big(\nu^n-P^*(v^n)\big)=0.\]
The measure $\nu^n-P^*(v^n)$ vanishes identically on $F$ and, by (ii) of Lemma~\ref{thm:lemaP}, $P^*(v^n)$ equals
$v^n(\theta)\delta_{p_\theta}$ on $I_\theta$; therefore:
\begin{equation}\label{eq:sumtheta}
\int_Kf\,\dd\big(\nu^n-P^*(v^n)\big)
=\sum_{\theta\in\Theta}\int_{I_\theta}f\,\dd\big(\nu^n_\theta-v^n(\theta)\delta_{p_\theta}\big).
\end{equation}
Since $(\nu^n_\theta)_{n\ge1}$ is weak*-null and $v^n(\theta)\to0$, we have that each term
of the sum on the righthand side of \eqref{eq:sumtheta} tends to zero. To conclude the proof, we will
show that, given $\varepsilon>0$, there exists a finite subset $\Phi\subset\Theta$ such that:
\[\Big\vert\sum_{\theta\in\Theta\setminus\Phi}
\int_{I_\theta}f\,\dd\big(\nu^n_\theta-v^n(\theta)\delta_{p_\theta}\big)\Big\vert<\varepsilon,\]
for all $n\ge1$. Write $I_\theta=[a_\theta,b_\theta]$ and $K=[a,b]$. Using that $f$ is increasing, we obtain:
\[\sum_{\theta\in\Theta}\big(f(b_\theta)-f(a_\theta)\big)\le f(b)-f(a)<+\infty;\]
in particular, there exists a finite subset $\Phi$ of $\Theta$ such that $f(b_\theta)-f(a_\theta)\le\varepsilon$, for $\theta\in\Theta\setminus\Phi$. Noting that $(\nu^n_\theta-v^n(\theta)\delta_{p_\theta})(I_\theta)=0$, Lemma~\ref{thm:inequalityintegral} yields:
\begin{align*}
\Big\vert\int_{I_\theta}f\,\dd\big(\nu^n_\theta-v^n(\theta)\delta_{p_\theta}\big)\Big\vert
&\le\frac12\big(f(b_\theta)-f(a_\theta)\big)\Vert\nu^n_\theta-v^n(\theta)\delta_{p_\theta}\Vert\\
&\le\big(f(b_\theta)-f(a_\theta)\big)\Vert\nu^n_\theta\Vert.
\end{align*}
Hence:
\[\Big\vert\sum_{\theta\in\Theta\setminus\Phi}
\int_{I_\theta}f\,\dd\big(\nu^n_\theta-v^n(\theta)\delta_{p_\theta}\big)\Big\vert\le
\sum_{\theta\in\Theta\setminus\Phi}\big(f(b_\theta)-f(a_\theta)\big)\Vert\nu^n_\theta\Vert\le
\varepsilon\sum_{\theta\in\Theta}\Vert\nu^n_\theta\Vert.\]
Having \eqref{eq:bounded} in mind, the conclusion follows.
\end{proof}

\begin{lem}\label{thm:lemadosr}
Let $(r^n_\gamma)_{n\ge1,\gamma\in\Gamma}$ be a family of real numbers, with $\Gamma$ a countable set.
Assume that $\lim_{n\to+\infty}r^n_\gamma=0$, for all $\gamma\in\Gamma$. Then, there exists an increasing
sequence $(\Phi_n)_{n\ge1}$ of finite subsets of $\Gamma$ such that $\Gamma=\bigcup_{n=1}^\infty\Phi_n$ and
$\lim_{n\to+\infty}\sum_{\gamma\in\Phi_n}r^n_\gamma=0$.
\end{lem}
\begin{proof}
If $\Gamma$ is finite, simply take $\Phi_n=\Gamma$, for all $n\ge1$. If $\Gamma$ is infinite, we can obviously
assume $\Gamma=\omega$. For each $k\ge1$, since $\lim_{n\to+\infty}\sum_{i=0}^kr^n_i=0$,
we can find $N_k\ge1$ such that $\big\vert\sum_{i=0}^kr^n_i\big\vert<\frac1k$ for $n\ge N_k$. Moreover,
we can assume that the sequence $(N_k)_{k\ge1}$ is strictly increasing. Now, for each $n\ge1$, let $\varphi(n)$
be the largest positive integer $k$ such that $N_k\le n$ (set $\varphi(n)=0$, if there is no such $k$).
Clearly $\varphi(n)$ increases with $n$ and, since $\varphi(N_k)\ge k$, we have $\lim_{n\to+\infty}\varphi(n)=+\infty$.
Thus, setting $\Phi_n=\big\{i\in\omega:i\le\varphi(n)\big\}$, we obtain that $(\Phi_n)_{n\ge1}$ is increasing and
$\bigcup_{n=1}^\infty\Phi_n=\omega$. Finally, given $n\ge1$, if $\varphi(n)\ne0$ then $N_{\varphi(n)}\le n$ and therefore:
\[\Big\vert\sum_{i\in\Phi_n}r^n_i\Big\vert=\Big\vert\sum_{i=0}^{\varphi(n)}r^n_i\Big\vert<\frac1{\varphi(n)},\]
proving that $\lim_{n\to+\infty}\sum_{i\in\Phi_n}r^n_i=0$.
\end{proof}

Finally, we can prove Theorem~\ref{thm:juntaLi}.
\begin{proof}[Proof of Theorem~\ref{thm:juntaLi}]
By Corollary~\ref{thm:eraremark}, it suffices to consider a normalized weak*-null sequence $(\mu^n)_{n\ge1}$
in $M(L)$ and obtain a weak*-null sequence $(\tilde\mu^n)_{n\ge1}$ in $M(K)$ such that $\phi_*\tilde\mu^n=\mu^n$,
for all $n\ge1$, and:
\[\limsup_{n\to+\infty}\Vert\tilde\mu^n\Vert\le\lambda.\]
Set $\mu^n_\gamma=\mu^n\vert_{L_\gamma}$;
since $L_\gamma$ is clopen in $L$, the sequence $(\mu^n_\gamma)_{n\ge1}$ is weak*-null in $M(L_\gamma)$.
Using the assumption that $(\phi_\gamma)^*C(L_\gamma)$ has the $\lambda$-\EP\ in $C(K_\gamma)$ and using
Lemma~\ref{thm:technical}, we obtain a weak*-null sequence $(\hat\mu^n_\gamma)_{n\ge1}$ in $M(K_\gamma)$
such that $(\phi_\gamma)_*\hat\mu^n_\gamma=\mu^n_\gamma$ and $\Vert\hat\mu^n_\gamma\Vert\le\lambda\Vert\mu^n_\gamma\Vert$,
for all $n\ge1$.

Let us first handle the trivial case in which $L=\bigcup_{\gamma\in\Gamma}L_\gamma$. In this case,
$K=\bigcup_{\gamma\in\Gamma}K_\gamma$ defines a finite partition of $K$ and
we obtain the measure $\tilde\mu^n\in M(K)$ by setting $\tilde\mu^n\vert_{K_\gamma}=\hat\mu^n_\gamma$, for all $\gamma\in\Gamma$.
Now assume that $L\setminus\bigcup_{\gamma\in\Gamma}L_\gamma$ contains a single point, which we denote by $\infty$.
Let $\Gamma_0$ be a countable subset of $\Gamma$ such that
$\mu^n_\gamma=0$ for all $n\ge1$ and $\gamma\in\Gamma\setminus\Gamma_0$.

Lemma~\ref{thm:colamedida} requires a family $(I_\theta)_{\theta\in\Theta}$ of disjoint clopen intervals. Thus, for each $\gamma\in\Gamma$,
we write the clopen subset $K_\gamma$ of $K$ as a finite disjoint union $K_\gamma=\bigcup_{j=1}^{m_\gamma}I_{\gamma,j}$
of clopen intervals $I_{\gamma,j}$ of $K$. Set:
\[\Theta=\big\{(\gamma,j):\gamma\in\Gamma,\ j=1,\ldots,m_\gamma\big\}.\]
Note that $\bigcup_{\theta\in\Theta}I_\theta=\bigcup_{\gamma\in\Gamma}K_\gamma$ is a proper subset of $K$, its complement
being $\phi^{-1}(\infty)$.
We now need a weak*-null sequence $(\nu^n_\theta)_{n\ge1}$ in $M(I_\theta)$. We start by defining a weak*-null
sequence $(\tilde\mu^n_\gamma)_{n\ge1}$ in $M(K_\gamma)$ and then we set $\nu^n_{\gamma,j}=(\tilde\mu^n_\gamma)\vert_{I_{\gamma,j}}$.

For $\gamma\in\Gamma_0$, set $r^n_\gamma=\sum_{j=1}^{m_\gamma}\vert\hat\mu^n_\gamma(I_{\gamma,j})\vert$. Since $(\hat\mu^n_\gamma)_{n\ge1}$ is weak*-null in $M(K_\gamma)$ and $I_{\gamma,j}$ is clopen in $K_\gamma$, we have that $\lim_{n\to+\infty}r^n_\gamma=0$. Applying Lemma~\ref{thm:lemadosr} to the family $(r^n_\gamma)_{n\ge1,\gamma\in\Gamma_0}$,
we obtain an increasing sequence $(\Phi_n)_{n\ge1}$ of finite subsets of $\Gamma_0$.
For $\gamma\in\Gamma$ and $n\ge1$, we define $\tilde\mu^n_\gamma\in M(K_\gamma)$ as follows.
If $\gamma\in\Phi_n$, we set $\tilde\mu^n_\gamma=\hat\mu^n_\gamma$; if $\gamma\in\Gamma\setminus\Phi_n$,
we take $\tilde\mu^n_\gamma$ to be a Hahn--Banach extension of $\mu^n_\gamma$. Note that, in any case,
we have:
\begin{equation}\label{eq:inanycase}
(\phi_\gamma)_*\tilde\mu^n_\gamma=\mu^n_\gamma
\end{equation}
and $\Vert\tilde\mu^n_\gamma\Vert\le\lambda\Vert\mu^n_\gamma\Vert$. More specifically:
\begin{equation}\label{eq:specifically}
\Vert\tilde\mu^n_\gamma\Vert\le\lambda\Vert\mu^n_\gamma\Vert,\ \text{for $\gamma\in\Phi_n$},\quad
\Vert\tilde\mu^n_\gamma\Vert=\Vert\mu^n_\gamma\Vert,\ \text{for $\gamma\in\Gamma\setminus\Phi_n$}.
\end{equation}
Note also that $(\tilde\mu^n_\gamma)_{n\ge1}$ is indeed a weak*-null sequence in $M(K_\gamma)$, for all $\gamma\in\Gamma$;
namely, the sequence is zero for $\gamma\not\in\Gamma_0$ and, for $\gamma\in\Gamma_0$, we have
$\tilde\mu^n_\gamma=\hat\mu^n_\gamma$, for $n$ sufficiently large.

We are now ready to apply Lemma~\ref{thm:colamedida}. Observe that:
\[\sum_{\theta\in\Theta}\Vert\nu^n_\theta\Vert=\sum_{\gamma\in\Gamma}\Vert\tilde\mu^n_\gamma\Vert
\le\lambda\sum_{\gamma\in\Gamma}\Vert\mu^n_\gamma\Vert\le\lambda\Vert\mu^n\Vert=\lambda,\]
so that assumption \eqref{eq:bounded} holds. Let $(\nu^n)_{n\ge1}$ be the weak*-null sequence in $M(K)$ given by the
lemma. It follows from (a) that:
\begin{equation}\label{eq:followsfroma}
\nu^n\vert_{K_{\gamma}}=\tilde\mu^n_\gamma.
\end{equation}
Moreover, it follows from (c) that:
\begin{equation}\label{eq:chata1}
\Vert\nu^n\Vert\le\sum_{\gamma\in\Gamma}\Vert\tilde\mu^n_\gamma\Vert
+\sum_{(\gamma,j)\in\Theta}\vert\tilde\mu^n_\gamma(I_{\gamma,j})\vert.
\end{equation}
We claim that $\limsup_{n\to+\infty}\Vert\nu^n\Vert\le\lambda$. Using \eqref{eq:specifically} we obtain:
\begin{gather}
\label{eq:chata2}\sum_{\gamma\in\Gamma}\Vert\tilde\mu^n_\gamma\Vert\le\lambda\sum_{\gamma\in\Phi_n}\Vert\mu^n_\gamma\Vert
+\sum_{\gamma\in\Gamma\setminus\Phi_n}\Vert\mu^n_\gamma\Vert,\\
\label{eq:chata3}\sum_{(\gamma,j)\in\Theta}\vert\tilde\mu^n_\gamma(I_{\gamma,j})\vert\le\sum_{\gamma\in\Phi_n}r^n_\gamma+
\sum_{\gamma\in\Gamma\setminus\Phi_n}\Vert\mu^n_\gamma\Vert.
\end{gather}
From \eqref{eq:chata1}, \eqref{eq:chata2}, \eqref{eq:chata3}, and the fact that $\lambda\ge2$, it follows that:
\[\Vert\nu^n\Vert\le\lambda\sum_{\gamma\in\Gamma}\Vert\mu^n_\gamma\Vert+\sum_{\gamma\in\Phi_n}r^n_\gamma
\le\lambda+\sum_{\gamma\in\Phi_n}r^n_\gamma.\]
Since $\lim_{n\to+\infty}\sum_{\gamma\in\Phi_n}r^n_\gamma=0$, the claim is proved.

Finally, pick $p\in\phi^{-1}(\infty)$ and set $\tilde\mu^n=\nu^n+\mu^n(L)\delta_p$. As $\mu^n(L)\to0$,
we have that $\limsup_{n\to+\infty}\Vert\tilde\mu^n\Vert\le\lambda$ and that $(\tilde\mu^n)_{n\ge1}$ is weak*-null.
It remains to check that $\phi_*\tilde\mu^n=\mu^n$, for all $n\ge1$. Note that, for all $\gamma\in\Gamma$:
\[(\phi_*\tilde\mu^n)\vert_{L_\gamma}=(\phi_*\nu^n)\vert_{L_\gamma}=(\phi_\gamma)_*(\nu^n\vert_{K_\gamma})
\stackrel{\text{\eqref{eq:followsfroma}}}=(\phi_\gamma)_*(\tilde\mu^n_\gamma)\stackrel{\text{\eqref{eq:inanycase}}}=
\mu^n_\gamma=\mu^n\vert_{L_\gamma}.\]
Thus $\phi_*\tilde\mu^n$ and $\mu^n$ agree on all Borel subsets of $\bigcup_{\gamma\in\Gamma}L_\gamma$.
Since the complement of $\bigcup_{\gamma\in\Gamma}L_\gamma$ in $L$ has just one point,
the proof will be concluded if we check that $(\phi_*\tilde\mu^n)(L)=\mu^n(L)$.
By (b) of Lemma~\ref{thm:colamedida},
$\nu^n(K)=0$ and hence:
\[(\phi_*\tilde\mu^n)(L)=\tilde\mu^n(K)=\mu^n(L).\qedhere\]
\end{proof}

\end{section}


\begin{thebibliography}{99}


\bibitem{ArgyrosLondon} S. A. Argyros, J. F. Castillo, A. S. Granero, M. Jim\'enez \& J. P. Moreno,
Complementation and embeddings of $c_0(I)$ in Banach spaces, {\em Proc.\ London.\ Math.\ Soc.}, 85 (3), 2002, pgs.\ 742---768.

\bibitem{Taras} T. Banakh \& A. Leiderman, Uniform Eberlein compactifications of metrizable spaces, {\em Topol.\ Appl.},
159 (7), 2012, pgs.\ 1691---1694.

\bibitem{JesusWeak} J. M. F. Castillo, M. Gonz\'alez \& P. L. Papini, On weak*-extensible Banach spaces, {\em Nonlinear analysis},
75, 2012, pgs.\ 4936---4941.

\bibitem{Corson} H. H. Corson, The weak topology of a Banach space, {\em Trans.\ Amer.\ Math.\ Soc.}, 101 (1), 1961, pgs.\ 1---15.

\bibitem{KK} O. Kalenda \& W. Kubi\'s, Complementation in spaces of continuous functions on compact lines, {\em J. Math.\ Anal.\ Appl.} 386, 2012, pgs.\ 241---257.

\bibitem{Kubis} W. Kubi\'s, Linearly ordered compacta and Banach spaces with a projectional resolution of the identity,
{\em Topology Appl.} 154 (3), 2007, pgs.\ 749---757.

\bibitem{Molto} A. Molt\'o, On a theorem of Sobczyk, {\em Bull.\ Aust.\ Math.\ Soc.}, 43 (1), 1991, pgs.\ 123---130.

\bibitem{Patterson} W. M. Patterson, Complemented $c_0$-subspaces of a non-separable $C(K)$-space,
{\em Canad.\ Math.\ Bull.} 36 (3), 1993, pgs.\ 351---357.

\bibitem{Phillips} R. S. Phillips, On linear transformations, {\em Trans.\ Amer.\ Math.\ Soc.}, 48, 1940,
pgs.\ 516---541.

\bibitem{WCGnother} H. P. Rosenthal, The heredity problem for weakly compactly generated Banach spaces,
{\em Compositio  Math.}, 28, 1974, pgs.\ 83---111.

\bibitem{c0sum} F. C. S\'anchez, Yet another proof of Sobczyk's theorem, Methods in Banach Space Theory,
Edited by J. M. F. Castillo \& W. B. Johnson, London Mathematical Society Lecture Note Series no.\ 337, Cambridge, 2006.


\bibitem{Sobczyk} A. Sobczyk, Projection of the space $(m)$ on its subspace $(c_0)$, {\em Bull.\ Amer.\
Math.\ Soc.}, 47, 1941, pgs.\ 938---947.

\bibitem{Veech} W. A. Veech, Short proof of Sobczyk's theorem, {\em Proc.\ Amer.\ Math.\ Soc.}, 28, 1971, pgs.\ 627---628.


\bibitem{WZQ} B. Wang, Y. Zhao \& W. Qian, On the weak-star extensibility, {\em Nonlinear Anal.}, 74, 2011, pgs.\ 2109---2215.

\bibitem{Zippin} M. Zippin, The embedding of Banach spaces into spaces with structure, {\em Illinois J. Math.}, 34 (3), 1990,
pgs.\ 586---606.

\end{thebibliography}
\end{document}